\documentclass[11pt]{amsart}
\usepackage{amsmath,amsfonts,amsthm,amssymb,amscd, verbatim, graphicx,color}
\usepackage[margin=2cm]{geometry}
\usepackage{hyperref}

\newcommand{\cycletype}[1]{\medmuskip=0mu
\thinmuskip=0mu
\thickmuskip=0mu 
\nulldelimiterspace=-1pt
\scriptspace=0pt
#1}

\renewcommand{\Gamma}{\varGamma}
\renewcommand{\epsilon}{\varepsilon}

\renewcommand{\leq}{\leqslant}
\renewcommand{\geq}{\geqslant}

\newcommand{\C}{\mathcal{C} }

\newcommand{\A}{\mathcal{A} }

\newcommand{\Z}{\mathbb{Z} }

\newtheorem{prop}{Proposition}[section]
\newtheorem{thm}[prop]{Theorem}

\newtheorem{cor}[prop]{Corollary}
\newtheorem{lem}[prop]{Lemma}

\theoremstyle{definition}

\numberwithin{equation}{section}
\begin{document}

\title{Abelian covers of alternating groups}

\author{Daniel Barrantes}
\address{Escuela de Matem\'aticas, Universidad de Costa Rica,  2060 San Jos\'e, Costa Rica.}
\email{daniel.barrantesgarbanzo@ucr.ac.cr}

\author{Nick Gill}\thanks{}
\address{Department of Mathematics,
University of South Wales,
Treforest, CF37 1DL, U.K.
}
\email{nicholas.gill@southwales.ac.uk}

\author{Jerem\'ias Ram\'irez}
\address{Escuela de Matem\'aticas, Universidad Nacional, Heredia, Costa Rica.}
\email{jeremias.ramirez.jimenez@una.cr}

\begin{abstract}
Let $G=A_n$, a finite alternating group. We study the commuting graph of $G$ and establish, for all possible values of $n$ barring $13, 14, 17$ and $19$, whether or not the independence number is equal to the clique-covering number.
\end{abstract}

\maketitle

\section{Introduction}

Throughout this paper $G$ is a finite group.

\subsection{Definitions and main result}\label{s: definitions}

Let $\Omega$ be a subset of the group $G$. A {\it non-commuting subset} of $\Omega$ is a subset $N$ of $\Omega$ such that
\[
x,y \in N \, \, \Longrightarrow xy\neq yx.
 \] 
 We define $\delta(\Omega)$ to be the maximum possible cardinality of a non-commuting subset of $\Omega$. A non-commuting subset that has cardinality $\delta(\Omega)$ will be called a {\it maximal non-commuting subset} of $\Omega$.
 
 An {\it abelian cover} of $\Omega$ is a set $\C$ of abelian subgroups of $G$ whose union contains $\Omega$. We define $\Delta(\Omega)$ to be the minimum possible cardinality of an abelian cover of $\Omega$. An abelian cover that has cardinality $\Delta(\Omega)$ will be called a {\it minimal abelian cover} of $\Omega$.  
    
 A simple application of the pigeon-hole principle yields the following fact:
  
\begin{lem}\label{l: basic inequality}
For any finite group $G$ and a subset $\Omega\subseteq G$, $\delta(\Omega) \leq \Delta(\Omega)$.
\end{lem} 
   
In this paper we study the case where $G$ is a finite alternating group; we are interested in ascertaining when $\delta(G)=\Delta(G)$. Our main result gives almost complete information.

\begin{thm}\label{t: alternating}
Let $A_n$ be the alternating group on $n$ letters, with $n$ a positive integer. 
\begin{enumerate}
\item If $n\leq 11$ or $n=15$, then $\delta(A_n)=\Delta(A_n)$;
\item If $n=12, 16$ or $18$ or $n\geq 20$, then $\delta(A_n)\neq\Delta(A_n)$.
\end{enumerate} 
\end{thm}

\subsection{The commuting graph}

The definitions just given can be recast in terms of a particular graph, as follows.

Let $G$ be a group and $\Omega \subseteq G$. Then the {\it commuting graph} of $\Omega$, denoted $\Gamma (\Omega)$, is the graph whose vertex set is $\Omega$ and with vertices connected if and only if they commute.

Now a cover of $\Omega$ by abelian subgroups of $G$ corresponds to a cover of $\Gamma( \Omega)$ by cliques, and so $\Delta(\Omega)$ is the {\it clique-covering number} of $\Gamma(\Omega)$. Similarly, a non-commuting subset of $\Omega$ is an independent set in $\Gamma (\Omega)$, and $\delta(\Omega)$ is the {\it independence number} of the graph.


\subsection{Notation}

Our notation is entirely standard. The {\it support} of a permutation $g \in S_n$ is the set $S$ of those elements in $\{1,2,3,\ldots, n\}$ satisfying $g(j)\neq j$. Two permutations $g,h \in S_n$ are {\it disjoint} if they have disjoint supports. Let $n_1,n_2, \ldots , n_r$ be positive integers such that $n_1+n_2+\cdots + n_r\leq n $. We say that a permutation $g \in S_n$ has cycle type $\cycletype{n_1-n_2-\cdots -n_r}$ if $g$ can be written as a product of disjoint cycles $g=g_1g_2\cdots g_r$, where the cycle $g_i$ is an $n_i$-cycle.

For $g\in G$, we write $Cl_G(g)$ for the conjugacy class of $g$ in $G$ (we will omit the subindex $G$, when the group $G$ is clear from the context.) Also we will write $C_G (g)$ for the centralizer of $g$ in $G$.



\subsection{Context}

The following result of Brown is a starting point for our research \cite{brown1, brown2}.

\begin{thm}\label{t: brown}
Let $S_n$ be the symmetric group on $n$ letters, with $n$ a positive integer. 
\begin{enumerate}
\item If $n\leq 7, n=9$ or $n=11$, then $\delta(S_n)=\Delta(S_n)$;
\item If $n=8$, $n=13$ or $n\geq 15$, then $\delta(S_n)\neq\Delta(S_n)$.
\end{enumerate} 
\end{thm} 

One can view our main result, Theorem~\ref{t: alternating}, as an analogue of Theorem~\ref{t: brown} for the alternating groups.

The study of groups via their commuting graph goes back many years. Perhaps the most famous result in this line of study concerns groups that are not finite: in 1976, B.\,H.~Neumann answered the following question of Erd\H{o}s from a few years earlier: if all non-commuting
sets in a group~${G}$ are finite, does there exist an upper bound ${n=n(G)}$ for the size of a non-commuting set in~${G}$ (i.e. is the value of $\delta(G)$ finite)? Neumann answered this question affirmatively by showing that if all non-commuting sets in a group~${G}$ are finite, then ${|G:Z(G)|}$ is finite
\cite{neumann}. Pyber subsequently gave a strong upper bound for ${|G:Z(G)|}$, subject to the same condition on ${G}$ (see \cite{pyber}).

The question of whether or not $\delta(G)=\Delta(G)$ has been studied by various authors for various groups $G$ (see, for instance, \cite{aamz, afo, ap, abg}). To our knowledge, Theorem~\ref{t: alternating} is the first result that asserts that $\delta(G)\neq \Delta(G)$ for some finite {\it simple} group $G$.

\subsection{Acknowledgments}

None of the results presented in this paper rely on computer calculations, however we are happy to acknowledge our use of the GAP computer package \cite{GAP}, which was a vital tool in the research process. In this direction, we particularly want to thank A.~Hulpke for help with coding at various stages.

All three authors would also like to thank M.~Josephy for valuable discussions.

\section{Background}\label{s: background}

In this section we record some basic results and definitions that will be needed in the sequel. The lemmas that we will need are little more than observations.

\begin{lem}\label{l: maximal}
There exists an abelian cover of $G$ of size $\Delta(G)$ and consisting entirely of maximal abelian subgroups.
\end{lem}
\begin{proof}
Let $\A$ be an abelian cover of $G$ of size $\Delta(G)$. Now replace every element of $\A$ by a maximal abelian subgroup of $G$ containing it. The resulting cover has the property we seek.
\end{proof}
  
\begin{lem}\label{l: centralizer}  
 Let $g\in G$, and suppose that $C_G(g)$ is abelian. Then
 \begin{enumerate}
\item $C_G(g)$ is a maximal abelian subgroup of $G$ and it is the unique maximal abelian subgroup of $G$ that contains $g$;
 \item There exists an abelian cover of $G$ of size $\Delta(G)$ containing $C_G(g)$.
 \end{enumerate}
\end{lem}
\begin{proof}
Let $A$ be a maximal abelian subgroup of $G$ that contains $C_G(g)$. In particular, $A$ contains $g$ and so, since $A$ is abelian, $A\leq C_G(g)$. Thus $A=C_G(g)$ and (1) is proved.

Let $\A$ be an abelian cover of $G$ of size $\Delta(G)$ and let $X$ be an element of $\A$ that contains $g$. By (1), $X$ is a subgroup of $C_G(g)$ and we can replace $X$ by $C_G(g)$ in $\A$ to obtain the cover that we need.
\end{proof}

Now we need a number of definitions; the ensuing lemmas will highlight their significance. Let $\C_a(G)$ be the set of all the abelian centralizers in $G$, $Y_a(G)$ its union and $Y_b(G)$ the complement of $Y_a(G)$. Hence
\begin{align*}
\C_a(G) &= \{ X\leq G \mid X\textrm{ is abelian, and } X=C_G(g) \textrm{ for some } g\in G\};  \\
Y_a(G) &= \bigcup_{X\in C_a(G)}X; \\
Y_b(G) &= G \setminus Y_a(G).
\end{align*}
We remark that both $Y_a(G)$ and $Y_{b}(G)$ are unions of conjugacy classes of $G$. 

Next, for every element $X\in \C_a(G)$, let $g_X\in G$ be such that $C_G(g_X)=X$. Now define
\[
N_a(G)=\{g_X \mid X\in \C_a(G)\}. 
\]
We caution that the set $N_a(G)$ is {\bf not} uniquely defined, since there may be more than one choice of $g_X$ for any given $X\in \C_a(G)$. In what follows we will refer to `a set $N_a(G)$', by which we will mean a set constructed in the given way. For all of the above definitions -- $\C_a(G), \C_b(G), Y_a(G), Y_b(G), N_a(G)$ -- when the group $G$ is obvious from the context, we may drop the $(G)$ from the name. Thus, for example, we will write $\C_a$ for $\C_a(G)$.

The next two lemmas connect the above definitions; the first is immediate.

\begin{lem}\label{l: Ya}
The set $\C_a(G)$ is an abelian cover of $Y_a(G)$, and the set $N_a(G)$ is a non-commuting subset of $Y_a(G)$. In particular, $\delta(Y_a(G))=\Delta(Y_a(G))$.
\end{lem}

\begin{lem}\label{l: get rid of abelian centralizers}
\begin{enumerate}
 \item There is an abelian cover of $G$ of size $\Delta(G)$ for which $\C_a(G)$ is a subset.
 \item Let $X$ be a set $N_a(G)$. There is a non-commuting subset of $G$ of size $\delta(G)$ for which $X$ is a subset.
\end{enumerate}
\end{lem}
\begin{proof}
The first statement is an immediate consequence of Lemma~\ref{l: centralizer}. For the second, suppose that $N$ is a non-commuting subset of $G$ and let $V=N\cap Y_a$. Clearly $|V|\leq \delta(Y_a(G))=|X|$, since $Y_a(G)$ is covered by $\delta(Y_a(G))$ abelian subgroups (Lemma~\ref{l: Ya}). Thus if we remove $V$ from $N$ and put $X$ in its place, then $N$ will not have diminished in size. What is more, by construction, $N$ is still non-commuting (since any element of $X$ does not commute with any element outside $Y_a$). We are done.
\end{proof}

\begin{lem}\label{l: leftover}
A group $G$ has an abelian cover of size $\delta(G)$ if and only if $\Delta(Y_{b}(G))=\delta(Y_{b}(G))$.
\end{lem}
\begin{proof}
Suppose, first, that $G$ has an abelian cover $\A$ of size $\delta(G)$. By Lemma~\ref{l: get rid of abelian centralizers} we can assume that $\C_a$ is a subset of $\A$. Let $N$ be a non-commuting subset of $G$ of size $\delta(G)$; then there is a subset, $N_a$, of $N$ of size $|\C_a|$ whose elements lie inside elements of $\C_a$. 

Now let $N_b=N\setminus N_a$, and $\C_b=\A\setminus \C_a$. Then $N_b$ is a non-commuting subset of $Y_{b}$, $\C_b$ is an abelian cover of $Y_{b}$, and $|N_b|=|\C_b|$, as required.

For the converse, suppose that $\Delta(Y_{b})=\delta(Y_{b})$, let $N_b$ be a maximal non-commuting subset of $Y_b$, and let $\C_b$ be a minimal abelian cover of $Y_b$. Now observe that $N_a\cup N_b$ is a maximal non-commuting subset of $Y_b$, and that $\C_a\cup \C_b$ is a minimal abelian cover of $Y_b$.
\end{proof}

We close by making an elementary remark. Suppose that $G$ is a group for which $\delta(G)=\Delta(G)$, that $N$ is a non-commuting subset of $G$ of size $\delta(G)$, and that $\C$ is an abelian cover of $G$ of size $\Delta(G)$. Then every element of $\C$ contains a unique element of $N$ (indeed the same is true for $G$ replaced in our suppositions by any subset $\Omega$ of $G$).

\section{Results on alternating groups}\label{s: symmetric}

We will make heavy use of the following elementary lemma, a proof of which can be found in \cite{brown2}.

\begin{lem}\label{l: brown}
Let $\sigma$ be a product of nontrivial disjoint cycles $\sigma_1 ,..., \sigma_k $ no two of which have the same length. Then every element of $S_n$ which commutes with $\sigma$ is a product $\tau \rho $, where $\tau$ is a product of powers of the cycles $\sigma_i$, $1 \leq i \leq k$ and $\rho$ and $\sigma$ have disjoint support.
\end{lem}

The treatment that follows is split into three cases -- when $n$ is even, when $n$ is odd, and when $n$ is small. The strategy in the first two cases is identical, and strongly resembles the method of Brown \cite{brown2}. However some of the details differ and so, for clarity, the cases are written separately.

\subsection{\texorpdfstring{$n$}{n} is even}

We assume here that $n$ is even, and let 
\[\sigma=(4,5)(6,7,8)(9,10,\ldots , n)\in A_n.\]
We are interested in the set $Cl_{A_n}(\sigma)$,  the conjugacy class of $\sigma$ in $A_n$ and we note first that an element of $Cl_{A_n}(\sigma)$ does not have an abelian centralizer in $A_n$.

\begin{lem}\label{l: alt even}
Suppose that $n=12$ or $n$ is even and $n\geq 16$, $g \in Cl_{A_n}(\sigma)$, $K$ is a maximal abelian subgroup of $A_n$ containing $g$, and $h$ is an element of $K$. Then one of the following occurs:
\begin{enumerate}
\item $h$ lies in an abelian centralizer.
\item $h$ has cycle type $\cycletype{2-3-\underbrace{d-d-\cdots -d}_k}$ where $1\leq d,k\leq n-8$ and $dk=n-8$.
\end{enumerate}
\end{lem}


\begin{proof}
It is sufficient to prove the result for the case  $g=\sigma=(4,5)(6,7,8)(9,...,n)$. Now we write $a=(4,5)$, $b=(6,7,8)$, $c=(9,10,...,n)$ and observe that $g=abc$.

The restrictions on $n$ mean that we can apply Lemma~\ref{l: brown} to the element $g$, and conclude that every permutation $h$ that commutes with $g$ will be of the form $a^i b^j c^k \rho$ where $0\leq i \leq 1$, $0\leq j \leq 2$, $0\leq k \leq n-8$, and the support of $\rho$ is a subset of $\{1,2,3\}$.

Observe that $h$ will satisfy condition (2) of the lemma if and only if one of the following holds:
\begin{itemize}
\item $i=1$, $j=1$, and $\rho=(1)$; or
\item$i=1$, $j=0$, and $\rho$ is a $3$-cycle;
\item$i=0$, $j=1$, and $\rho$ is a $2$-cycle.
\end{itemize}

In the first column below we list a representative $h$ of every $S_n$-conjugacy class that has a non-trivial intersection with $K$, and does not satisfy condition (2) of the lemma. In the second column we list a permutation $f$ that, provided $n\neq 12$, has abelian centralizer, and which commutes with $h$.

\begin{center} 
\begin{tabular}{l l}
Permutation $h$ & Permutation $f$ \\
\hline
$(4,5)(9,...,n)^k$ & $(1,6,2,7,3,8)(9,...,n)$ \\
$(6,7,8)(9,...,n)^k$ & $(1,2,3,4)(6,7,8)(9,...,n).$ \\
$(1,2)(4,5)(9,...,n)^k$ & $(1,4,2,5)(6,7,8)(9,...,n)$ \\
$(1,2,3)(6,7,8)(9,...,n)^k$ & $(1,6,2,7,3,8)(9,...,n)$ \\
$(1,2)(4,5)(6,7,8)(9,...,n)^k$ & $(1,4,2,5)(6,7,8)(9,..,n)$ \\
$(1,2,3)(4,5)(6,7,8)(9,...,n)^k$ & $(1,6,2,7,3,8)(9,...,n)$ \\
\end{tabular}     
\end{center}

We caution that not all of the permutations $h$ listed above lie in $K$ since, for certain values of $k$, the listed permutation will not be even. However the given list certainly includes all of the required permutations. Furthermore it is easy to check, using Lemma~\ref{l: brown}, that $C_G(f)$ is indeed abelian. 

We must deal with the case when $n=12$. Again we list a representative $h$ of every $S_{12}$-conjugacy class that has a non-trivial intersection with $K$, and does not satisfy condition (2) of the lemma. Again, in the second column we list a permutation $f$ that has abelian centralizer, and which commutes with $h$.

\begin{center}
\begin{tabular}{l l}
Permutation $h$ & Permutation $f$ \\
\hline
$(4,5)(9,10,11,12)$ & $(1,6,2,7,3,8)(9,10,11,12)$ \\
$(6,7,8)(9,11)(10,12)$ & $(1,2,3,4,5)(6,7,8)(9,11)(10,12)$ \\
$(1,2)(4,5)(9,11)(10,12)$ & $(3,6,7,8)(1,9,4,10,2,11,5,12)$ \\
$(1,2,3)(6,7,8)(9,11)(10,12)$ & $(1,6,2,7,3,8)(9,10,11,12)$ \\
$(1,2)(4,5)(6,7,8)(9,11)(10,12)$ & $(6,7,8)(1,4,9,2,5,11)(10,12)$ \\
$(1,2,3)(4,5)(6,7,8)(9,10,11,12)$ & $(1,6,2,7,3,8)(9,10,11,12)$ \\
\end{tabular}     
\end{center}
A quick computation checks that the centralizer for $(1,2,3,4,5)(6,7,8)(9,11)(10,12)$ is abelian, and by lemma \ref{l: brown} the rest are obviously abelian as well.
\end{proof}

We need an easy lemma concerning permutations that satisfy condition (2) of the previous lemma.

\begin{lem}\label{l: powers not needed}
Suppose that $n=12$, or $n$ is even and $n\geq 16$, and that $h\in A_n$ has cycle type $\cycletype{2-3-\underbrace{d-d-\cdots -d}_k}$ where $1\leq d,k\leq n-8$ and $dk=n-8$. Then $C_{A_n}(h)$ contains a permutation $f$ with cycle type $\cycletype{2-3-(n-8)}$ and, moreover, $C_{A_n}(f) \leq C_{A_n}(h)$.
\end{lem}
\begin{proof}
Write $h=abc_1$ where $a$ is a 2-cycle, $b$ is a 3-cycle, $c_1$ is a product of $k$ $d$-cycles, and $a,b$ and $c$ have disjoint support. Then $c_1$ is a power of an $\cycletype{(n-8)}$-cycle, $c$, and $f=abc\in C_{A_n}(h)$ has cycle type $\cycletype{2-3-(n-8)}$.

Now the restrictions on $n$ mean that we can apply Lemma~\ref{l: brown} to the element $f$ to conclude that any element $f_1$ that commutes with $f$ will be of the form $a^i b^j c^k \rho$ where $0\leq i \leq 1$, $0\leq j \leq 2$, $0\leq k \leq n-8$, and  $\rho$ and $f$ have disjoint supports. Now one can check directly that such an element commutes with $h$, and we conclude that $C_{A_n}(f) \leq C_{A_n}(h)$.
\end{proof}

 \begin{prop}\label{p: alt even}
Suppose that $n=12$ or $n$ is even and $n\geq 16$. Then $\delta(A_n)\neq \Delta(A_n)$.
\end{prop}
\begin{proof}
Assume, for a contradiction, that $\delta(A_n)=\Delta(A_n)$.
Let $E$ be a maximal set of non-commuting elements in $A_n$. By Lemma~\ref{l: get rid of abelian centralizers}, we can (and do) assume that $E$ contains a set $N_a(A_n)$. Let $\mathcal{A}$ be a minimal abelian cover of $A_n$ By Lemma~\ref{l: get rid of abelian centralizers}, we can (and do) assume that $\mathcal{A}$ contains $\C_a$. As before, the set $Cl_{A_n}(\sigma)$ is the conjugacy class of $\sigma=(1,2)(3,4,5)(9,10,\ldots,n)$ in $A_n$, and every element in $Cl_{A_n}(\sigma)$ have cycle type $2-3-(n-8)$.

Now let $\gamma$ be a fixed $(n-8)$-cycle.  For $g\in Cl_{A_n}(\sigma)$, write ${\rm Long}(g)$ for the $(n-8)$-cycle that forms part of the cycle decomposition of $g$. Now define
\begin{align*}
Cl_\gamma &:= \{g\in Cl_{A_n}(\sigma) \mid {\rm Long}(g)=\gamma\}; \\
\mathcal{A}_\gamma &:=\{ A\in \mathcal{A} \mid h\in A \textrm{ for some } h\in Cl_\gamma\}; \\
E_\gamma &:= E \cap Cl_\gamma.
\end{align*}

Since $\delta(A_n)=\Delta(A_n)$, every element of $\mathcal{A}$ contains an element of $E$.
Let $A\in \mathcal{A}_\gamma$ and let $h$ be the unique element in $A\cap E$. Then, since $\mathcal{A}$ contains $\mathcal{C}_a$, Lemma~\ref{l: alt even} implies that $h$ has cycle type $2-3-\underbrace{d-d-\cdots -d}_k$ where $1\leq d,k\leq n-8$ and $dk=n-8$. Now Lemma~\ref{l: powers not needed} implies that we can replace $h$ in $E$ by an element $f\in A$ with cycle type $2-3-(n-8)$, and $E$ will still be non-commuting. Hence we can (and do) assume that, for every $A\in \mathcal{A}_\gamma$ the unique element in $A\cap E$ is of type $2-3-(n-8)$.

Since $\bigcup\limits_{A\in \mathcal{A}_\gamma}A\supseteq Cl_\gamma$ we conclude that $|E_\gamma|=|\mathcal{A}_\gamma|$ and $E_\gamma$ is precisely the set of elements in $E$ that lie in some element of $\mathcal{A}_\gamma$.

Now one can check, firstly, that $|Cl_\gamma|=1120$ and, secondly, that if $\Omega\subset Cl_\gamma$ is a set of commuting permutations, then $|\Omega|\leq 4$. So that every element $A\in \mathcal{A}_\gamma$ contains at most 4 permutations of the given type and, in consequence, $|\mathcal{A}_\gamma|\geq 1120/4=280$. Thus $|E_\gamma|\geq 280$.

Now consider $\mathcal{B}$, the set of groups generated by 4 disjoint cycles $a,b,c, \gamma$ where $a$ is a $2-$cycle, both $b$ and $c$ are $3-$cycles and $\gamma$ is as above. We can easily see that $\mathcal{B}$ is an abelian cover of $\C_\gamma$. What is more $| \mathcal{B} |=280$. Therefore $|E_\gamma|\leq 280$ and so $|E_\gamma|=280=|\mathcal{A}_\gamma|$.

Now let $A\in\mathcal{A}_\gamma$. Since $|\mathcal{A}_\gamma|=280$ and $A$ contains at most $4$ elements of $Cl_\gamma$ which has size $1120$, we see that $A$ contains exactly $4$ elements of $Cl_\gamma$. Let $g$ be an element of $Cl_\gamma$ and write $g=ab\gamma$ where $a$ is a 2-cycle, $b$ is a 3-cycle, and $a,b$ and $\gamma$ are disjoint. One can easily check that $g$ lies in two maximal abelian subgroups:
\begin{itemize}
\item $A_1=\langle a,b,\gamma , \rho_1 \rangle$ where $\rho_1$ is a $2$-cycle that is disjoint from $a,b$ and $\gamma$ ;
\item $A_2=\langle a,b, \gamma, \rho_2 \rangle$ where $\rho_2$ is a $3$-cycle that is disjoint from $a,b$ and $\gamma$;
\end{itemize}
Now $A_1$ contains $2$ elements of $Cl_\gamma$, while $A_2$ contains $4$ elements of $Cl_\gamma$. We conclude that $A_2\in \mathcal{A}_\gamma$. It now follows that $A_\gamma$ is equal to the set $\mathcal{B}$, defined in the previous paragraph.

Now observe that the following groups all lie in $\mathcal{A}_\gamma$:
\begin{center}
\begin{tabular}{cc}
$A_1 :=\langle(1,2,3),(4,5,6),(7,8),\gamma\rangle;$ &
$A_2 := \langle(1,2,3),(5,7,8),(4,6),\gamma\rangle;$ \\
$A_3:= \langle(5,7,8),(2,4,6),(1,3),\gamma\rangle;$ &
$A_4:=\langle(2,4,6),(1,3,5),(7,8),\gamma\rangle;$ \\
$A_5:=\langle(1,3,5),(6,7,8),(2,4),\gamma\rangle;$ &
$A_6:=\langle(1,2,3),(4,7,8),(5,6),\gamma\rangle;$ \\
$A_7:=\langle(4,7,8),(1,5,6),(2,3),\gamma\rangle;$ &
$A_8:=\langle(2,3,4),(1,5,6),(7,8),\gamma\rangle;$ \\
$A_9:=\langle(2,3,4),(6,7,8),(1,5),\gamma\rangle.$ &
\end{tabular}
\end{center}

Let $a_i$ be the unique element in $A_i\cap E$ for $i=1,\dots,9$. Without loss of generality, we may assume that $a_1=(1,2,3)(7,8)\gamma$. Now for $a_2$ not to commute with $a_1$ it must be of the form $(5,7,8)^i(4,6)\gamma^j$, where $i=1,2$ and $j$ and $n-8$ are coprime. Notice that the choice of $i$ and $j$ does not affect the set of permutations in $Cl_\gamma$ that commute with it. Thus, without lost of generality we may assume that $a_2=(5,7,8)(4,6)\gamma$.

Following the same logic, we may take $a_3=(2,4,6)(1,3)\gamma$, $a_4=(1,3,5)(7,8)\gamma$ and $a_5=(6,7,8)(2,4)\gamma$. Now starting from the fact that $a_1=(1,2,3)(7,8)\gamma$, and following the same logic as above, we can also deduce that $a_6=(4,7,8)(5,6)\gamma$ , $a_7=(1,5,6)(2,3)\gamma$ and $a_8=(2,3,4)(7,8)\gamma$. 

Now we find that we cannot choose $a_9$, since any element we choose will commute either with $a_8=(2,3,4)(7,8)\gamma$ or $a_5=(6,7,8)(2,4)\gamma$. We have the contradiction that we sought.
\end{proof}

\subsection{\texorpdfstring{$n$}{n} is odd}

The case where $n$ is odd will be proven in essentially the same way as the even case but considering instead permutations of cyclic type $\cycletype{2-3-8-(n-16)}$ since the permutations of cyclic type $\cycletype{2-3-(n-8)}$ aren't in $A_n$ when $n$ is odd. Let $\tau=(4,5)(6,7,8)(9,10,\ldots,16)(17,18,\ldots , n)\in A_n$. We consider the set $Cl_{A_n}(\tau )$ and observe, as before, that an element of $Cl_{A_n}(\tau )$ does not have an abelian centralizer in $A_n$. We will have the same situation as in the even case:

\begin{lem}\label{l: alt odd}
Suppose that $n$ is odd and $n\geq 21$, $g \in Cl_{A_n}(\tau)$, $K$ is a maximal abelian group containing $g$, and $h$ and element of $K$. Then one of the following occurs:
\begin{enumerate}
\item $h$ lies in an abelian centralizer.
\item $h$ has cycle type $\cycletype{2-3- \underbrace{d-d- \cdots -d}_k - \underbrace{e-e- \cdots -e}_j}$ where $1 \leq d,k \leq 8$, $1 \leq e,j \leq \cycletype{n-16}$, $dk=8$, and $ej=n-16$.
\end{enumerate}
\end{lem}
\begin{proof}
Again it is sufficient to choose 
\[g=\tau=(4,5)(6,7,8)(9,10,11,12,13,14,15,16)(17,18,...,n)\] and check that every permutation that commutes with $g$ satisfies the lemma. By Lemma ~\ref{l: brown} every permutation that commutes with $g$ will be of the form $a^i b^j c^k d^l \rho$ with $a=(4,5)$, $b=(6,7,8)$, $c=(9,10,...,16)$, $d=(17,18,...,n)$ and the support of $\rho$ is $\lbrace1,2,3 \rbrace$.

As in the even case we list a representative $h$ of every $S_n$-conjugacy class that has a non-trivial intersection with $K$, and does not satisfy condition (2) of the lemma. In the second column we list a permutation $f$ that, provided $n\neq 19$, has abelian centralizer, and which commutes with $h$.

\begin{center}

\begin{tabular}{l l}
Permutation $h$ & Permutation $f$ \\
\hline
$(4,5)(9,10,...,16)^k(17,...,n)^l$ & $(1,6,2,7,3,8)(9,10,...,16)(17,...,n)$ \\
$(6,7,8)(9,10,...,16)^k(17,...,n)^l$ & $(1,4,2,5)(6,7,8)(9,10,...,16)(17,...,n)$ \\ 
$(1,2)(4,5)(9,10,...,16)^k(17,...,n)^l$ & $(1,4,2,5)(6,7,8)(9,10,...,16)(17,...,n)$ \\
$(1,2,3)(6,7,8)(9,10,...,16)^k(17,...,n)^l$ & $(1,6,2,7,3,8)(9,10,...,16)(17,...,n)$ \\
$(1,2)(4,5)(6,7,8)(9,10,...,16)^k(7,...,n)^l$ & $(1,4,2,5)(6,7,8)(9,10,...,16)(17,..,n)$ \\
$(1,2,3)(4,5)(6,7,8)(9,10,...,16)^k(17,...,n)^l$ & $(1,6,2,7,3,8)(9,10,...,16)(17,...,n)$ \\

\end{tabular}     

\end{center}
\end{proof}

Now to deal with the permutations that satisfy condition (2) we will use the next lemma. The statement and proof are analogous to Lemma~\ref{l: powers not needed}, and so the proof is omitted.

\begin{lem}\label{l: odd powers not needed}
Suppose that $n$ is odd and $n\geq 21$, and that $h\in A_n$ has cycle type $\cycletype{2-3- \underbrace{d-d- \cdots -d}_k - \underbrace{e-e- \cdots -e}_j}$ where $1 \leq d,k \leq 8$, $1 \leq e,j \leq n-16$, $dk=8$, and $ej=n-16$. Then $C_{A_n}(h)$ contains a permutation $f$ with cycle type $\cycletype{2-3-8-(n-16)}$ and, moreover, $C_{A_n}(f) \leq C_{A_n}(h)$.
\end{lem}

\begin{prop}\label{p: alt odd}
Suppose that $n$ is odd and $n \geq 21$. Then $\delta(A_n)\neq \Delta(A_n)$.
\end{prop}
\begin{proof}
We do essentially the same thing as in the even case. Assume $\delta(A_n)= \Delta(A_n)$. Let $E$ a non commuting subset of size $\delta(A_n)$ such that $N_a(A_n) \subset E$. Let $Cl_{\gamma , \theta}$ be the set of permutations of cycle type $\cycletype{2-3-8-(n-16)}$ where the $8$-cycle is a given cycle $\gamma$ and the $\cycletype{(n-16)}$-cycle is a given cycle $\theta$, with the support of $\gamma$ and $\theta$ disjoint from $ \lbrace 1,2,3,4,5,6,7,8 \rbrace$. Define
\begin{align*}
\mathcal{A}_{\gamma , \theta} &:=\{ A\in \mathcal{A} \mid h\in A \textrm{ for some } h\in Cl_{\gamma , \theta} \};\\
E_{\gamma , \theta} &:= E \cap Cl_ {\gamma , \theta }.
\end{align*}
Again we see that to cover $Cl_{\gamma , \theta}$ we need at least $280$ abelian groups, so that $|E_{\gamma, \theta}|=|A_{\gamma, \theta}| \geq 280$. By Lemma~\ref{l: alt odd} these permutations will satisfy condition (2) of said lemma. By Lemma ~\ref{l: odd powers not needed} we can take these elements to be of the type $\cycletype{2-3-8-(n-16)}$.
 
Now consider $\mathcal{B}$ to be the set of groups generated by five disjoint cycles $a,b,c,\gamma,\theta$ where $a$ is a $2$-cycle, $b$ and $c$ are $3$-cycles, and $\gamma$, $\theta$ as above. This is a cover of $Cl_{\gamma, \theta}$ by $280$ abelian groups and so each one must contain exactly one element of $E_{\gamma, \theta}$.

Now we assume that $(1,2,3)(7,8)\gamma \theta $ is in $E_{\gamma, \theta}$ and consider in order the representatives for 
\begin{center}
 \begin{tabular}{cc}
 $\langle(1,2,3),(4,5,6),(7,8),\gamma,\theta\rangle$; & $\langle(1,2,3),(5,7,8),(4,6),\gamma,\theta\rangle$; \\
 $\langle(5,7,8),(2,4,6),(1,3),\gamma,\theta\rangle$; & $\langle(2,4,6),(1,3,5),(7,8),\gamma,\theta\rangle$; \\
 $\langle(1,3,5),(6,7,8),(2,4),\gamma,\theta\rangle$; & $\langle(1,2,3),(4,7,8),(5,6),\gamma,\theta\rangle$; \\
 $\langle(4,7,8),(1,5,6),(2,3),\gamma,\theta\rangle$; & $\langle(2,3,4),(1,5,6),(7,8),\gamma,\theta\rangle$.
 \end{tabular}
\end{center}
As before, we may assume that the following elements are in $E_{\gamma,\theta}$:
\begin{center}
 \begin{tabular}{cc}
 $(5,7,8)(4,6)\gamma \theta$; & $(2,4,6)(1,3)\gamma \theta$; \\
 $(1,3,5)(7,8)\gamma \theta$; & $(6,7,8)(2,4)\gamma \theta$; \\
 $(4,7,8)(5,6)\gamma \theta$; & $(1,5,6)(2,3)\gamma \theta$; \\
 $(2,3,4)(7,8) \gamma \theta$.
 \end{tabular}
\end{center}
 Now we must choose a representative for $\langle(2,3,4),(6,7,8),(1,5),\gamma ,\theta\rangle$, but we can not do this, since any element we choose will commute either with $(2,3,4)(7,8)\gamma \theta$ or $(6,7,8)(2,4)\gamma \theta$ both of which are already in $E_{\gamma, \theta}$, this contradiction implies that $\delta(A_n) \neq \Delta(A_n)$.
\end{proof}
 
\subsection{Small \texorpdfstring{$n$}{n}}

To complete our understanding of the situation in alternating groups, we need to consider the cases $n\leq 11$ and $n=13,14,15, 19$. We will not give full information here -- to prove Theorem~\ref{t: alternating} it is sufficient to prove that $\delta(A_n)=\Delta(A_n)$ for $n\leq 11$ and $n=15$.

\begin{lem}\label{l: alt abelian centralizer}
Let $n$ be a positive integer. Then every element of $A_n$ lies in an abelian centralizer if and only if $n\leq 7$ or $n=10$.

Moreover if $n=8$ or $9$, then the only elements of $A_n$ that do not lie in an abelian centralizer have cycle type $\cycletype{4-4}$. Similarly, the only elements of $A_{11}$ that do not lie in an abelian centralizer have cycle type $\cycletype{4-4-3}$.
\end{lem}
\begin{proof}
Suppose that $n\leq 7$ or that $n=10$. One can check directly that, for all $g\in A_n$, there exists $h\in A_n$ such that $g\in C_{A_n}(h)$ and $C_{A_n}(h)$ is abelian. The statements for $n=8,9$ and $11$ can also be checked by direct computation.

Suppose now that $n\geq 12$. One can check that if $n$ is even (resp. odd), then elements of type $\cycletype{4-4-(n-9)}$ (resp. $\cycletype{4-4-(n-8)}$) do not lie in an abelian centralizer in $A_n$.
\end{proof}

In the notation of \textsection\ref{s: background}, the lemma asserts that, for $n\leq 7$ and $n=10$, we have $Y_a(A_n)=A_n$. Now Lemma~\ref{l: Ya} immediately yields the following corollary.

\begin{cor}\label{c: alt abelian centralizer}
If $n\leq 7$ or $n=10$, then $\delta(A_n)=\Delta(A_n)$.
\end{cor}

The following lemmas will help us deal with some of the remaining cases.

\begin{lem}\label{2k-2k-l elements}
Let $n$ be a positive integer with $n\geq 8$, and let $k$ and $\ell$ be positive integers such that $n-(4k+ \ell)<\min(\ell,2k)$ and $\ell\neq 2k$, $\ell \neq 4k$. Let $\sigma, \tau \in S_n$ of cycle type $\cycletype{2k-2k-\ell}$ such that $\sigma \tau=\tau \sigma$. Then,
\begin{enumerate}
\item  The $\ell$-cycle of $\tau$ is a power of the $\ell$-cycle of $\sigma$. 
\item  The supports of $\sigma$ and $\tau$ are equal.
\end{enumerate}   
\end{lem}

\begin{proof}
For the first part, let $r \in \{1,2, \ldots , n\}$, we consider the orbit of $r$ under $\langle \sigma\rangle$, this is $Orb_{\sigma}(r)=\{\sigma^i(r):i\in \mathbb{Z}\}$. Suppose that $|Orb_{\sigma}(r)|=\ell$. We have that $\sigma^i\tau^j(r)=\tau^j\sigma^i(r)$ for all $i,j \in \mathbb{Z}$, and  so
\[Orb_{\sigma}(\tau^j(r))=\tau^j\left(Orb_{\sigma}(r)\right),\]
which implies that
\[Orb_{\sigma}(\tau^j(r))=Orb_{\sigma}(r),\]
because there is only one orbit by $\sigma$ of size $\ell$. We conclude that $\tau^j(r)\in Orb_{\sigma}(r)$  for all $j \in \mathbb{Z}$, and then $Orb_{\tau}(r)\subseteq Orb_{\sigma}(r)$. 

Since $\tau$ has cycle type $\cycletype{2k-2k-\ell}$ there are three possibilities: $|Orb_{\tau}(r)|=\ell$, $|Orb_{\tau}(r)|=1$ and $|Orb_{\tau}(r)|=2k$. In the first case clearly $Orb_{\tau}(r)=Orb_{\sigma}(r)$, so that both $\ell$-cycles have the same support. Lemma~\ref{l: brown} then tells us that one $\ell$-cycle must be the power of the other.

Suppose one of the other possibilities holds: then if $|Orb_{\tau}(r)=1|$ we have that $\tau(r)=r$, then $\tau \sigma^j(r)=\sigma^j\tau(r)=\sigma^j(r)$, so $\sigma^j(r)$ is fixed by $\tau$ for all $j=1,2,3,\ldots, \ell$, i.e. $\tau$ has $\ell$ fixed points. But this is impossible, because $n-(4k+\ell)<\ell$. 

This means that we must have $|Orb_{\tau}(r)|=2k$ for all $r$'s in the support of $\sigma$'s $\ell$-cycle. Since $\tau$ has cycle type $\cycletype{2k-2k-\ell}$ the support of $\sigma$'s $\ell$-cycle must be the same as the support for one of the $2k$-cycles of $\sigma$, or the  union of the supports of the two $2k$-cycles of $\sigma$. So we have only two possibilities: $\ell=2k$ or $\ell=4k$, neither one possible by hypothesis. This ends the proof of the first part.

For the second part, let $r \in \{1,2,3,\ldots , n\}$, and suppose (for a contradiction) that $\tau(r)=r$ and  $\sigma(r)\neq r$. This implies that $Orb_{\sigma}(r)$ must have $2k$ or $\ell$ elements. But now observe that, for all $i$,
\[
 \tau\sigma^i(r) = \sigma^i\tau(r) = \sigma^i(r)
\]
and we conclude that every element in $Orb_{\sigma}(r)$ is fixed by $\tau$. But this is impossible, because $\tau$ only fixes $n-(4k+\ell)$ elements. We conclude that every element fixed by $\tau$ is fixed by $\sigma$. The same reasoning yields that every element fixed by $\sigma$ is fixed by $\tau$ and the result holds.
\end{proof}

\begin{lem}\label{l: double even class}
Let $n$ a positive integer with $n\geq 8$, and suppose that $n=4k$ for some integer $k$. Let $\Lambda$ be the conjugacy class of elements of cycle type $\cycletype{2k-2k}$ in $A_n$. Then the commutator graph $\Gamma(\Lambda)$ is disconnected. Furthermore if $\Omega$ is (the vertex set of) a connected component of $\Gamma(\Lambda)$, then $\Omega\subset H < S_n$, with $H\cong S_k \wr K_4$, where $K_4$ is the normal Klein $4$-subgroup of $S_4$.
\end{lem}
\begin{proof}
Let $g=(1,2,\dots, 2k)(2k+1,2k+2,\dots, 4k)$ and consider the following partition of the set $\{1,2,\dots, n\}$ into four subsets of size $k$:
\begin{center}
\begin{tabular}{ll} 
$\{1,3,5,\dots, 2k-1\}$, \, & $\{2k+1, 2k+3, 2k+5, \dots, 4k-2\},$  \\
$\{2,4,6, \dots, 2k\}$,  & $\{2k+2, 2k+4, 2k+6, \dots, 4k\}.$
\end{tabular}
\end{center}
It is easy to see that this is the unique partition of type $\cycletype{k-k-k-k}$ on which $g$ acts as a double-transposition. What is more direct computation confirms that the same is true of any element of $\Lambda$ that commutes with $g$. Let $\Omega$ be the connected component of $\Gamma(\Lambda)$ that contains $g$. We conclude that $\Omega$ lies inside the subgroup generated by all elements of $\Lambda$ that act as double-transpositions on $\Lambda$. All such elements lie inside $H$ and we are done.
\end{proof}

\begin{cor}\label{2k+2k and l}
Let $n$ be a positive integer, $n\geq 8$, and let $k,\ell$ be positive integers such that $n-(4k+\ell)<\min(2k,\ell)$. Let $\Lambda$ be the conjugacy class of elements of cycle type $\cycletype{2k-2k-\ell}$ in $A_n$. Then the commutator graph $\Gamma(\Lambda)$ is disconnected. Furthermore if $\Omega$ is (the vertex set of) a connected component of $\Gamma(\Lambda)$, then $\Omega\subset H < S_n$ with $H\cong (S_k \wr K_4)\times \mathbb{Z}/\ell\mathbb{Z}$.
\end{cor}
\begin{proof}
Let $\sigma, \tau \in \Omega$, the vertex set of a connected component of $\Gamma(\Lambda)$. By Lemma \ref{2k-2k-l elements} the $\ell$-cycle of $\tau$ is a power of the $\ell$-cycle of $\sigma$, and their $\cycletype{2k-2k}$ parts have the same support. By Lemma \ref{l: double even class} $\Omega \subset H$ with $H\cong (S_k\wr K_4) \times \mathbb{Z}/\ell\mathbb{Z}$. 
	\end{proof}

\begin{cor}\label{c: alt less 11}
If $n\leq 11$, then $\delta(A_n)=\Delta(A_n)$.
\end{cor}
\begin{proof}
By Corollary~\ref{c: alt abelian centralizer}, we may (and we do) assume that $n=8$, $9$ or $11$. By Lemma~\ref{l: leftover}, we must prove that $\delta(Y_b(A_n))=\Delta(Y_b(A_n))$ in each case. By Lemma~\ref{l: alt abelian centralizer}, $Y_b(A_n)$ is equal to the conjugacy class of cycle type $\cycletype{4-4}$ for $n=8$ or $9$, and $\cycletype{4-4-3}$ for $n=11$.

Now, by Corollary~\ref{2k+2k and l}, the commutator graph of $Y_b(A_n)$ is disconnected and each connected component lies inside a subgroup of $S_n$ isomorphic to $S_2\wr K_4$ for $n=8$ or $n=9$, and $(S_2\wr K_4 )\times \mathbb{Z}/3\mathbb{Z}$ for $n=11$. Let $\Omega$ be such a connected component; clearly if we can show that $\delta(\Omega)=\Delta(\Omega)$, then we are done.

We consider first the cases $n=8$ and $n=9$, without loss of generality, we take $\Omega$ to be a connected component inside ``the'' wreath product $H=S_2\wr S_4$. Direct computation reveals that $H$ contains $12$ elements of cycle type $4-4$; they are the following elements and their inverses:

\begin{center}
\begin{tabular}{l l l}
 $(1,3,2,4)(5,7,6,8),$ & $(1,5,2,6)(3,7,4,8),$ & $(1,7,2,8)(3,5,4,6),$  
 \\
 $(1,3,2,4)(5,8,6,7),$ & $(1,5,2,6)(3,8,4,7),$ & $(1,7,2,8)(3,6,4,5).$
 \end{tabular} 
\end{center}

Now observe that the three elements in the first row form a non-commuting set, while the elements in each column generate an abelian group. Taking $\Omega$ to be this set of $12$ elements, we conclude that $\delta(\Omega)=\Delta(\Omega)=3$, and we are done.

For the case $n=11$, we observe that $H$ has $24$ elements of cycle type $4-4-3$ namely

\begin{center}
	\begin{tabular}{l l l}
		$(1,3,2,4)(5,7,6,8)(9,10,11),$ & $(1,5,2,6)(3,7,4,8)(9,10,11),$ & $(1,7,2,8)(3,5,4,6)(9,10,11),$  
		\\
		$(1,3,2,4)(5,8,6,7)(9,10,11),$ & $(1,5,2,6)(3,8,4,7)(9,10,11),$ & $(1,7,2,8)(3,6,4,5)(9,10,11),$
		\\
		$(1,3,2,4)(5,7,6,8)(9,11,10),$ & $(1,5,2,6)(3,7,4,8)(9,11,10),$ & $(1,7,2,8)(3,5,4,6)(9,11,10),$  
		\\
		$(1,3,2,4)(5,8,6,7)(9,11,10),$ & $(1,5,2,6)(3,8,4,7)(9,11,10),$ & $(1,7,2,8)(3,6,4,5)(9,11,10).$
		
	\end{tabular} 
\end{center}

and their inverses. Again, observe that the three elements in the first row form a non-commuting set, while the elements in each column generate an abelian group. Taking $\Omega$ to be this set of $24$ elements, we conclude that $\delta(\Omega)=\Delta(\Omega)=3$, and we are done.
\end{proof}

\begin{lem}\label{l: alt 15}
$\delta (A_{15})=\Delta (A_{15})$.
\end{lem}
\begin{proof}
Direct computation shows that $Y_b (A_{15})$ consists of the elements of cycle type $\cycletype{7-4-4}$, $\cycletype{6-6-3}$, $\cycletype{6-4-2-2}$, and $\cycletype{6-3-3-2}$. Notice that among these classes no two elements in different classes commute. Hence we may set $\Lambda$ to be each class in turn, and show that in eacy case $\delta(\Lambda)=\Delta(\Lambda)$.

{\bf Class $\cycletype{7-4-4}$}: In this case Corollary~\ref{2k+2k and l} implies that $\Gamma(\Lambda)$ is disconnected. Taking $\Omega$ to be the maximal connected component that contains $(1,3,2,4)(5,7,6,8)(9,10,11,12,13,14,15)$ and setting $\sigma = (9,10,11,12,13,14,15)$ one can check that this component has $72$ elements, as follows:
\begin{center}
\begin{tabular}{l l l}
 $(1,3,2,4)(5,7,6,8)\sigma ^i ,$ & $(1,5,2,6)(3,7,4,8)\sigma ^i ,$ & $(1,7,2,8)(3,5,4,6)\sigma ^i ,$  
 \\
 $(1,3,2,4)(5,8,6,7)\sigma ^i ,$ & $(1,5,2,6)(3,8,4,7)\sigma ^i ,$ & $(1,7,2,8)(3,6,4,5)\sigma ^i .$
 \end{tabular} 
\end{center}
(Here $i$ can take any value between $1$ and $6$, and an element is either on the list or its inverse is on the list.)

Now the elements of each column generate an abelian group and these $3$ groups cover $\Omega$. On the other hand, fixing $i$ and taking the elements of any row we get a non-commuting set of size $3$, as required. 

{\bf Class $\cycletype{6-6-3}$}: Reasoning along the lines of Corollary~\ref{2k+2k and l}, we can deduce that $\Gamma(\Lambda)$ is disconnected in this situation, with a maximal connected component lying inside a subgroup of $S_n$ isomorphic to $(S_3\wr K_4) \times \Z/3\Z \times \Z/3\Z$. Further a maximal connected component contains $432$ elements, with one example as follows:
\begin{tiny}
\begin{center}
\begin{tabular}{l l l}
$( 1,10, 2,11, 3,12)( 4, 8, 5, 9, 6, 7)\sigma^i,$&$( 1, 8, 3, 7, 2, 9)( 4,11, 6,10, 5,12)\sigma^i,$&$( 1, 4, 2, 5, 3, 6)( 7,10, 8,11, 9,12)\sigma^i,$\\$( 1,10, 3,12, 2,11
 )( 4, 8, 6, 7, 5, 9)\sigma^i,$&$( 1, 7, 3, 9, 2, 8)( 4,11, 6,10, 5,12)\sigma^i,$&$( 1, 4, 3, 6, 2, 5)( 7,11, 9,10, 8,12)\sigma^i,$\\$( 1, 4, 3, 6, 2, 5)( 7,10, 9,12, 8,11
 )\sigma^i,$&$( 1, 7, 2, 8, 3, 9)( 4,10, 5,11, 6,12)\sigma^i,$&$( 1,11, 3,10, 2,12)( 4, 7, 6, 9, 5, 8)\sigma^i,$\\$( 1, 5, 3, 4, 2, 6)( 7,12, 9,11, 8,10)\sigma^i,$&$
( 1, 8, 3, 7, 2, 9)( 4,12, 6,11, 5,10)\sigma^i,$&$( 1,10, 2,11, 3,12)( 4, 9, 5, 7, 6, 8)\sigma^i,$\\$( 1, 4, 2, 5, 3, 6)( 7,12, 8,10, 9,11)\sigma^i,$&$( 1, 7, 3, 9, 2, 8)
( 4,12, 6,11, 5,10)\sigma^i,$&$( 1,10, 3,12, 2,11)( 4, 9, 6, 8, 5, 7)\sigma^i,$\\$( 1,11, 3,10, 2,12)( 4, 9, 6, 8, 5, 7)\sigma^i,$&$( 1, 7, 2, 8, 3, 9)( 4,12, 5,10, 6,11
 )\sigma^i,$&$( 1, 5, 3, 4, 2, 6)( 7,10, 9,12, 8,11)\sigma^i,$\\$( 1,11, 3,10, 2,12)( 4, 8, 6, 7, 5, 9)\sigma^i,$&$( 1, 7, 2, 8, 3, 9)( 4,11, 5,12, 6,10)\sigma^i,$&$
( 1, 4, 2, 5, 3, 6)( 7,11, 8,12, 9,10)\sigma^i,$\\$( 1,10, 3,12, 2,11)( 4, 7, 6, 9, 5, 8)\sigma^i,$&$( 1, 7, 3, 9, 2, 8)( 4,10, 6,12, 5,11)\sigma^i,$&$( 1, 5, 3, 4, 2, 6)
( 7,11, 9,10, 8,12)\sigma^i,$\\$( 1,10, 2,11, 3,12)( 4, 7, 5, 8, 6, 9)\sigma^i,$&$( 1, 8, 3, 7, 2, 9)( 4,10, 6,12, 5,11)\sigma^i,$&$( 1, 4, 3, 6, 2, 5)( 7,12, 9,11, 8,10
 )\sigma^i,$\\$( 1, 5, 3, 4, 2, 6)( 7,11, 8,12, 9,10)\sigma^i,$&$( 1, 7, 3, 8, 2, 9)( 4,12, 6,10, 5,11)\sigma^i,$&$( 1,11, 2,10, 3,12)( 4, 9, 5, 8, 6, 7)\sigma^i,$\\$
( 1, 5, 2, 6, 3, 4)( 7,11, 9,10, 8,12)\sigma^i,$&$( 1, 7, 2, 9, 3, 8)( 4,12, 5,11, 6,10)\sigma^i,$&$( 1,10, 2,12, 3,11)( 4, 9, 5, 8, 6, 7)\sigma^i,$\\$( 1, 4, 2, 5, 3, 6)
( 7,10, 9,12, 8,11)\sigma^i,$&$( 1, 7, 2, 9, 3, 8)( 4,10, 5,12, 6,11)\sigma^i,$&$( 1,10, 3,11, 2,12)( 4, 9, 6, 7, 5, 8)\sigma^i,$\\$( 1, 5, 3, 4, 2, 6)( 7,12, 8,10, 9,11
 )\sigma^i,$&$( 1, 8, 2, 7, 3, 9)( 4,11, 5,10, 6,12)\sigma^i,$&$( 1,11, 2,10, 3,12)( 4, 8, 5, 7, 6, 9)\sigma^i,$\\$( 1, 5, 2, 6, 3, 4)( 7,12, 9,11, 8,10)\sigma^i,$&$
( 1, 7, 3, 8, 2, 9)( 4,10, 6,11, 5,12)\sigma^i,$&$( 1,10, 2,12, 3,11)( 4, 8, 5, 7, 6, 9)\sigma^i,$\\$( 1, 4, 2, 5, 3, 6)( 7,12, 9,11, 8,10)\sigma^i,$&$( 1, 8, 2, 7, 3, 9)
( 4,10, 5,12, 6,11)\sigma^i,$&$( 1,10, 3,11, 2,12)( 4, 7, 6, 8, 5, 9)\sigma^i,$\\$( 1, 7, 3, 8, 2, 9)( 4,11, 6,12, 5,10)\sigma^i,$&$( 1, 4, 2, 5, 3, 6)( 7,11, 9,10, 8,12
 )\sigma^i,$&$( 1,10, 3,11, 2,12)( 4, 8, 6, 9, 5, 7)\sigma^i,$\\$( 1, 8, 2, 7, 3, 9)( 4,12, 5,11, 6,10)\sigma^i,$&$( 1, 5, 2, 6, 3, 4)( 7,10, 9,12, 8,11)\sigma^i,$&$
( 1,10, 2,12, 3,11)( 4, 7, 5, 9, 6, 8)\sigma^i,$\\$( 1, 7, 2, 9, 3, 8)( 4,11, 5,10, 6,12)\sigma^i,$&$( 1, 5, 3, 4, 2, 6)( 7,10, 8,11, 9,12)\sigma^i,$&$( 1,11, 2,10, 3,12)
( 4, 7, 5, 9, 6, 8)\sigma^i,$\\$( 1, 7, 2, 8, 3, 9)( 4,10, 6,12, 5,11)\sigma^i,$&$( 1,10, 3,11, 2,12)( 4, 8, 5, 7, 6, 9)\sigma^i,$&$( 1, 4, 2, 6, 3, 5)( 7,10, 8,12, 9,11
 )\sigma^i,$\\$( 1, 7, 3, 9, 2, 8)( 4,10, 5,11, 6,12)\sigma^i,$&$( 1,11, 3,12, 2,10)( 4, 8, 5, 7, 6, 9)\sigma^i,$&$( 1, 4, 3, 5, 2, 6)( 7,10, 9,11, 8,12)\sigma^i,$\\$
( 1, 8, 3, 7, 2, 9)( 4,11, 5,12, 6,10)\sigma^i,$&$( 1,11, 2,10, 3,12)( 4, 7, 6, 8, 5, 9)\sigma^i,$&$( 1, 4, 3, 5, 2, 6)( 7,12, 9,10, 8,11)\sigma^i,$\\$( 1, 7, 2, 8, 3, 9)
( 4,12, 6,11, 5,10)\sigma^i,$&$( 1,10, 3,11, 2,12)( 4, 9, 5, 8, 6, 7)\sigma^i,$&$( 1, 6, 3, 4, 2, 5)( 7,11, 9,12, 8,10)\sigma^i,$\\$( 1, 7, 3, 9, 2, 8)( 4,12, 5,10, 6,11
 )\sigma^i,$&$( 1,11, 3,12, 2,10)( 4, 9, 5, 8, 6, 7)\sigma^i,$&$( 1, 4, 2, 6, 3, 5)( 7,12, 8,11, 9,10)\sigma^i,$\\$( 1, 8, 3, 7, 2, 9)( 4,12, 5,10, 6,11)\sigma^i,$&$
( 1,11, 2,10, 3,12)( 4, 9, 6, 7, 5, 8)\sigma^i,$&$( 1, 6, 3, 4, 2, 5)( 7,12, 9,10, 8,11)\sigma^i,$\\$( 1, 8, 3, 7, 2, 9)( 4,10, 5,11, 6,12)\sigma^i,$&$( 1,11, 2,10, 3,12)
( 4, 8, 6, 9, 5, 7)\sigma^i,$&$( 1, 4, 2, 6, 3, 5)( 7,11, 8,10, 9,12)\sigma^i,$\\$( 1, 7, 3, 9, 2, 8)( 4,11, 5,12, 6,10)\sigma^i,$&$( 1,11, 3,12, 2,10)( 4, 7, 5, 9, 6, 8
 )\sigma^i,$&$( 1, 6, 3, 4, 2, 5)( 7,10, 9,11, 8,12)\sigma^i,$\\$( 1, 7, 2, 8, 3, 9)( 4,11, 6,10, 5,12)\sigma^i,$&$( 1,10, 3,11, 2,12)( 4, 7, 5, 9, 6, 8)\sigma^i,$&$
( 1, 4, 3, 5, 2, 6)( 7,11, 9,12, 8,10)\sigma^i,$\\$( 1, 7, 2, 9, 3, 8)( 4,10, 6,11, 5,12)\sigma^i,$&$( 1,10, 3,12, 2,11)( 4, 9, 5, 7, 6, 8)\sigma^i,$&$( 1, 4, 2, 6, 3, 5)
( 7,10, 9,11, 8,12)\sigma^i,$\\$( 1,12, 3,11, 2,10)( 4, 9, 5, 7, 6, 8)\sigma^i,$&$( 1, 7, 3, 8, 2, 9)( 4,10, 5,12, 6,11)\sigma^i,$&$( 1, 4, 3, 5, 2, 6)( 7,10, 8,12, 9,11
 )\sigma^i,$\\$( 1, 7, 3, 8, 2, 9)( 4,11, 5,10, 6,12)\sigma^i,$&$( 1,12, 2,10, 3,11)( 4, 9, 6, 8, 5, 7)\sigma^i,$&$( 1, 6, 3, 4, 2, 5)( 7,12, 8,11, 9,10)\sigma^i,$\\$
( 1,10, 3,12, 2,11)( 4, 7, 5, 8, 6, 9)\sigma^i,$&$( 1, 9, 3, 7, 2, 8)( 4,12, 5,11, 6,10)\sigma^i,$&$( 1, 4, 2, 6, 3, 5)( 7,11, 9,12, 8,10)\sigma^i,$\\$( 1, 7, 2, 9, 3, 8)
( 4,11, 6,12, 5,10)\sigma^i,$&$( 1,12, 3,11, 2,10)( 4, 7, 5, 8, 6, 9)\sigma^i,$&$( 1, 4, 3, 5, 2, 6)( 7,11, 8,10, 9,12)\sigma^i,$\\$( 1, 9, 3, 7, 2, 8)( 4,11, 5,10, 6,12
 )\sigma^i,$&$( 1,12, 2,10, 3,11)( 4, 8, 6, 7, 5, 9)\sigma^i,$&$( 1, 6, 3, 4, 2, 5)( 7,11, 8,10, 9,12)\sigma^i,$\\$( 1,12, 2,10, 3,11)( 4, 7, 6, 9, 5, 8)\sigma^i,$&$
( 1, 7, 2, 9, 3, 8)( 4,12, 6,10, 5,11)\sigma^i,$&$( 1, 6, 3, 4, 2, 5)( 7,10, 8,12, 9,11)\sigma^i,$\\$( 1, 9, 3, 7, 2, 8)( 4,10, 5,12, 6,11)\sigma^i,$&$( 1,12, 3,11, 2,10)
( 4, 8, 5, 9, 6, 7)\sigma^i,$&$( 1, 4, 3, 5, 2, 6)( 7,12, 8,11, 9,10)\sigma^i,$\\$( 1, 7, 3, 8, 2, 9)( 4,12, 5,11, 6,10)\sigma^i,$&$( 1,10, 3,12, 2,11)( 4, 8, 5, 9, 6, 7
 )\sigma^i,$&$( 1, 4, 2, 6, 3, 5)( 7,12, 9,10, 8,11)\sigma^i,$ 
  \end{tabular} 
\end{center}
\end{tiny}
(Here $i$ can take any value between $1$ and $6$, and an element is either on the list or its inverse is on the list.)

In the list above the elements of every row generate an abelian group and so these $36$ groups cover the connected component. On the other hand, fixing $i$ and taking the elements of the first column we get a non-commuting set of size $36$, as required.

{\bf Class $\cycletype{6-4-2-2}$}: a maximal connected component in $\Gamma(\Lambda)$ contains $72$ pemutations. One example is as follows:

\begin{center}
\begin{tabular}{l l l}
$(1,2,3,4)(5,8)(6,7)\sigma^i,$&$(1,2,3,4)(5,6)(7,8)\sigma^i,$&$(1,2,3,4)(5,7)(6,8)\sigma^i,$\\$(1,2)(3,4)(5,6,7,8)\sigma^i,$&$(1,3)(2,4)(5,6,7,8)\sigma^i,$&$(1,4)
(2,3)(5,6,7,8)\sigma^i,$\\$(1,2,4,3)(5,6)(7,8)\sigma^i,$&$(1,2,4,3)(5,7)(6,8)\sigma^i,$&$(1,2,4,3)(5,8)(6,7)\sigma^i,$\\$(1,3)(2,4)(5,6,8,7)\sigma^i,$&$(1,4)(2,3)
(5,6,8,7)\sigma^i,$&$(1,2)(3,4)(5,6,8,7)\sigma^i,$\\$(1,3,2,4)(5,7)(6,8)\sigma^i,$&$(1,3,2,4)(5,8)(6,7)\sigma^i,$&$(1,3,2,4)(5,6)(7,8)\sigma^i,$\\$(1,4)(2,3)
(5,7,6,8)\sigma^i,$&$(1,2)(3,4)(5,7,6,8)\sigma^i,$&$(1,3)(2,4)(5,7,6,8)\sigma^i,$
\end{tabular} 
\end{center}
(Here $\sigma=(9,10,11,12,13,14,15)$, $i$ can be either $1$ or $5$, and an element is either on the list or its inverse is on the list.)

In the list above the elements of every row generate an abelian group and so these $3$ groups cover the connected component. On the other hand, fixing $i$ and taking the elements of the first column we get a non-commuting set of size $6$, as required.

{\bf Class $\cycletype{6-3-3-2}$}:  a maximal connected component in $\Gamma(\Lambda)$ contains $96$ pemutations. One example is as follows:
\begin{tiny}
\begin{center}
\begin{tabular}{l l l l}
$( 1, 2, 5, 6, 3, 4)( 7, 8, 9)(10,12,11),$&$( 1, 2, 5, 6, 3, 4)( 7, 8, 9)(10,12,11),$&$( 1, 5, 3)( 2, 6, 4)( 7,12, 8,11, 9,10),$&$( 1, 2, 5, 6, 3, 4)
( 7, 9, 8)(10,11,12),$\\$( 1, 3, 5)( 2, 4, 6)( 7,11, 9,12, 8,10),$&$( 1, 4, 5, 2, 3, 6)( 7, 8, 9)(10,12,11),$&$( 1, 5, 3)( 2, 6, 4)( 7,11, 9,12, 8,10
 ),$&$( 1, 4, 5, 2, 3, 6)( 7, 9, 8)(10,11,12),$\\$( 1, 3, 5)( 2, 4, 6)( 7,11, 8,10, 9,12),$&$( 1, 2, 3, 4, 5, 6)( 7, 8, 9)(10,12,11),$&$
( 1, 5, 3)( 2, 6, 4)( 7,11, 8,10, 9,12),$&$( 1, 2, 3, 4, 5, 6)( 7, 9, 8)(10,11,12),$\\$( 1, 4, 3, 2, 5, 6)( 7, 8, 9)(10,12,11),$&$( 1, 3, 5)
( 2, 6, 4)( 7,11, 8,10, 9,12),$&$( 1, 4, 3, 2, 5, 6)( 7, 9, 8)(10,11,12),$&$( 1, 5, 3)( 2, 4, 6)( 7,11, 8,10, 9,12),$\\$( 1, 2, 3, 6, 5, 4)( 7, 8, 9)
(10,12,11),$&$( 1, 3, 5)( 2, 6, 4)( 7,11, 9,12, 8,10),$&$( 1, 2, 3, 6, 5, 4)( 7, 9, 8)(10,11,12),$&$( 1, 5, 3)( 2, 4, 6)( 7,11, 9,12, 8,10
 ),$\\$( 1, 2, 5, 4, 3, 6)( 7, 8, 9)(10,12,11),$&$( 1, 3, 5)( 2, 6, 4)( 7,12, 8,11, 9,10),$&$( 1, 2, 5, 4, 3, 6)( 7, 9, 8)(10,11,12),$&$
( 1, 5, 3)( 2, 4, 6)( 7,12, 8,11, 9,10),$\\$( 1, 3, 5)( 2, 6, 4)( 7,10, 9,12, 8,11),$&$( 1, 4, 3, 2, 5, 6)( 7, 8, 9)(10,11,12),$&$( 1, 5, 3)
( 2, 4, 6)( 7,10, 9,12, 8,11),$&$( 1, 4, 3, 2, 5, 6)( 7, 9, 8)(10,12,11),$\\$( 1, 3, 5)( 2, 6, 4)( 7,12, 9,11, 8,10),$&$( 1, 2, 5, 4, 3, 6)( 7, 8, 9)
(10,11,12),$&$( 1, 5, 3)( 2, 4, 6)( 7,12, 9,11, 8,10),$&$( 1, 2, 5, 4, 3, 6)( 7, 9, 8)(10,12,11),$\\$( 1, 3, 5)( 2, 6, 4)( 7,12, 8,10, 9,11
 ),$&$( 1, 2, 3, 6, 5, 4)( 7, 8, 9)(10,11,12),$&$( 1, 5, 3)( 2, 4, 6)( 7,12, 8,10, 9,11),$&$( 1, 2, 3, 6, 5, 4)( 7, 9, 8)(10,12,11),$\\$
( 1, 2, 5, 6, 3, 4)( 7, 8, 9)(10,11,12),$&$( 1, 3, 5)( 2, 4, 6)( 7,10, 9,12, 8,11),$&$( 1, 2, 5, 6, 3, 4)( 7, 9, 8)(10,12,11),$&$( 1, 5, 3)( 2, 6, 4)
( 7,10, 9,12, 8,11),$\\$( 1, 4, 5, 2, 3, 6)( 7, 8, 9)(10,11,12),$&$( 1, 3, 5)( 2, 4, 6)( 7,12, 8,10, 9,11),$&$( 1, 4, 5, 2, 3, 6)( 7, 9, 8)
(10,12,11),$&$( 1, 5, 3)( 2, 6, 4)( 7,12, 8,10, 9,11),$\\$( 1, 2, 3, 4, 5, 6)( 7, 8, 9)(10,11,12),$&$( 1, 3, 5)( 2, 4, 6)( 7,12, 9,11, 8,10
 ),$&$( 1, 2, 3, 4, 5, 6)( 7, 9, 8)(10,12,11),$&$( 1, 5, 3)( 2, 6, 4)( 7,12, 9,11, 8,10).$
\end{tabular} 
\end{center}
\end{tiny}
(Here, to save space, we have omitted the cycle $(13,14)$ from every permutation. As always every element is either on the list or its inverse is on the list.)

In the list above the elements of every row generate an abelian group and so these $12$ groups cover the connected component. On the other hand, taking the elements of the first column we get a non-commuting set of size $12$, as required.
\end{proof}

Theorem~\ref{t: alternating} now follows from Propositions~\ref{p: alt even} and \ref{p: alt odd}, Corollary~\ref{c: alt less 11} and Lemma~\ref{l: alt 15}.

\section{Further work}

Note that both Theorems~\ref{t: alternating} and \ref{t: brown} have numbers for which information is not given. In particular, the following groups are not covered: $S_{10}, S_{12}, S_{14}, A_{13}, A_{14}, A_{17}$ and $A_{19}$.
 
Results from this paper can be used to study these seven groups: for instance, one can check that $Y_b(S_{10})$ consists of the single conjugacy class $\cycletype{4-2-2}$. Thus Lemma~\ref{l: leftover} and a consideration of the abelian groups that contain elements from this class imply that $\delta(S_{10})=\Delta(S_{10})$ if and only if one can find $9450$ noncommuting permutations in $S_{10}$ of cycle type $\cycletype{4-2-2}$.

In \cite{brown1, brown2}, Brown studies the asymptotics of the sequence $\Delta(S_n)/\delta(S_n)$ and, in particular, shows that this sequence has no limit and takes on infinitely many distinct values arbitrarily close to $1$. It seems reasonable to think that the same is true of the sequence $\Delta(A_n)/\delta(A_n)$, but this has not yet been established.

Graphs analogous to the commuting graph have been studied in various contexts. In particular, in \cite{abg}, a graph $\Gamma_c(G)$ is defined for any finite group $G$ and any $c\in \Z^+\cup\{\infty\}$ as follows: vertices are the elements of $G$, with two vertices $a,b\in G$ joined by an edge if and only if $\langle a,b\rangle$ is nilpotent of class at most $c$. Observe that the commuting graph is simply the graph $\Gamma_1(G)$. 

A natural extension to the work in the current paper would be to establish whether or not $\delta(\Gamma_c(G))=\Delta(\Gamma_c(G))$ for $G$ alternating or symmetric, and $c\neq 1$.

\providecommand{\bysame}{\leavevmode\hbox to3em{\hrulefill}\thinspace}
\providecommand{\MR}{\relax\ifhmode\unskip\space\fi MR }
\providecommand{\MRhref}[2]{%
  \href{http://www.ams.org/mathscinet-getitem?mr=#1}{#2}
}
\providecommand{\href}[2]{#2}

\end{document}